\documentclass[12pt]{amsart}

\usepackage{amssymb}

\usepackage{enumitem}

\usepackage[unicode, pdftex]{hyperref}

\makeatletter
\@namedef{subjclassname@2020}{%
  \textup{2020} Mathematics Subject Classification}
\makeatother

\usepackage[T1]{fontenc}

\newtheorem{theorem}{Theorem}[section]
\newtheorem{corollary}[theorem]{Corollary}
\newtheorem{lemma}[theorem]{Lemma}
\newtheorem{proposition}[theorem]{Proposition}
\newtheorem{problem}[theorem]{Problem}

\theoremstyle{definition}

\newtheorem{example}[theorem]{Example}

\begin{document}
\sloppy

\baselineskip=17pt


\title[On stable pairs of Hahn and extremal sections \dots]{On stable pairs of Hahn and extremal sections of separately continuous functions on the products with a scattered multiplier}

\author[O. Maslyuchenko]{Oleksandr Maslyuchenko}
\address{Institute of Mathematics\\ University of Silesia in Katowice\\
Bankowa 12\\
40-007 Katowice, Poland \ \ \ \textit{and}\ \ \ 
Faculty of Mathematics and Informatics\\ Yuriy Fedkovych Chernivtsi National University\\ Kotsiubynskoho 2\\ 58012 Chernivtsi\\ Ukraine}
\email{ovmasl@gmail.com}

\author[A. Lianha]{Anastasiia Lianha}
\address{Faculty of Mathematics and Informatics\\ Yuriy Fedkovych Chernivtsi National University\\ Kotsiubynskoho 2\\ 58012 Chernivtsi\\ Ukraine}
\email{lianha.anastasiia@chnu.edu.ua}

\date{2.01.2025}

\begin{abstract}
The minimal and the maximal sections $\wedge_f,\vee\!_f:X\to\overline{\mathbb R}$ of a function $f:X\times Y\to\overline{\mathbb R}$ are defined by
$\wedge_f(x)=\inf\limits_{y\in Y}f(x,y)$  and $\vee\!\!_f(x)=\sup\limits_{y\in Y}f(x,y)$  for any $x\in X$.
 A pair $(g,h)$ of functions on $X$ is called a stable pair of Hahn if there exists a sequence of continuous functions $u_n$ on $X$ such that $h(x)=\min\limits_{n\in\mathbb{N}}u_n(x)$ and $g(x)=\max\limits_{n\in\mathbb{N}}u_n(x)$ for any $x\in X$. Evidently, every stable pair of Hahn is a countable pair of Hahn, and hence a pair of Hahn.
We prove that for any separately continuous function $f$ on the product of compact spaces $X$ and $Y$ such that $Y$ is scattered and at least one of them has the countable chain property, the pair $(\wedge_f,\vee\!_f)$ is a stable pair of Hahn.
We prove that for any stable pair of Hahn $(g,h)$ on the product of a topological space $X$ and an infinity completely regular space $Y$ there exists a separately continuous function $f$ on $X\times Y$ such that $\wedge_f=g$ and $\vee\!_f=h$.

\end{abstract}

\subjclass[2020]
{Primary 
	54C30, 
	26A15, 
	26B99; 
Secondary
	54G12, 
	54H05, 
	26A21.
}

\keywords{pair of Hahn, stable pair of Hahn, countable pair of Hahn, semicontinuous function, separately continuous function, scattered space, extremal section, minimal section, maximal section}

\maketitle

\section{Introduction}

At the beginning of  20th century H.~Hahn proved in \cite{H} that for any pair of functions on a metric space such that the least of them is upper semicontinuous and the greatest of them is lower semicontinuous, there is an intermediate continuous function. Later, in the middle of the 20th century,  Hahn's result was generalized to a wider classes of spaces (Dieudonne for paracompact spaces, Tong and Katetov for normal spaces). Such a pair of functions V.~K.~Maslyuchenko in \cite{MMV} proposed to call a pair of Hahn. More precisely, a pair  $(g,h)$ of functions $g,h:X\to\overline{\mathbb R}$ on a topological space $X$ is called \emph{a pair of Hahn} if $g\le h$, $g$ is  an upper semicontinuous function and $h$ is a lower semicontinuous function.
In \cite{MMV}  the authors consider   \emph{the minimal and the maximal sections} $\wedge_f,\vee\!_f:X\to\overline{\mathbb R}$ of a function  $f:X\times Y\to\overline{\mathbb R}$ which are defined by 
$$\wedge_f(x)=\inf\limits_{y\in Y}f(x,y)\mbox{ and }\vee\!\!_f(x)=\sup\limits_{y\in Y}f(x,y)\mbox{ for any }x\in X.$$
They observed that  $(\wedge_f,\vee\!_f)$ is a pair of Hahn for an arbitrary separately continuous function and asked the following question.
\begin{problem}\label{prb1}
Let $X$ and $Y$ be topological spaces and $(g,h)$ be a pair of Hahn on $X$. Under what assumptions does  there exist a separately continuous function ${f\colon X\times Y\to \overline{\mathbb R}}$ such that $g=\wedge_f$ and $h=\vee\!_f$?
\end{problem}
In \cite{MMV} this problem was solved only for $X=[a;b]$ and $Y=[c;d]$. The authors proved that every pair of Hahn on a segment $[a;b]$ is generated by a separately continuous function $f:[a;b]\times [c;d]\to \overline{\mathbb R}$. They also obtained a partial answer to Problem~\ref{prb1} in more general situation: if $X=Y$ is a perfectly normal space with $G_\delta$ diagonal and normal square and $(g,h)$ is a pair of Hahn on $X$ such that $g$ is continuous, then there is a separately continuous function $f:X^2\to \overline{\mathbb R}$ such that $g\le f\le h$ and $\vee\!_f=h$. They used a method of the construction of a separately continuous function with a given diagonal which is similar to \cite{MMS}.

In our previous paper \cite{MK} we solved Problem~\ref{prb1} for a perfectly normal space $X$ and a topological space $Y$ which has a non-scattered compactification. Moreover, we introduced the notion of \textit{a countable pair of Hahn} as a pair $(g,h)$ such that $g(x)=\inf\limits_{n\in\mathbb N}g_n(x)\le\sup\limits_{n\in\mathbb N}h_n(x)=h(x)$, $x\in X$, where $g_n$ and $h_n$ are continuous functions, and obtained the following result. 
\begin{theorem}[{\cite[Theorems 7 and 10]{MK}}] Let $X$ be a topological space, $Y$ be a topological space which has a non-scattered compactification and $(g,h)$ be a countable pair of Hahn on $X$. Then there exists a separately continuous function ${f:X\times Y\to\overline{\mathbb R}}$ such that $g=\wedge_f$ and $h=\vee\!_f$.
\end{theorem}

 Obviously, every countable pair of Hahn is a pair of Hahn.
Furthermore, by the Tong theorem \cite{T} the notion of a countable pair of Hahn
coincides with the notion of a pair of Hahn for perfectly normal space. Moreover, using Tong's result we may suppose that in the definition of a countable pair of Hahn the functions $g_n$ and $h_n$ are equal. Therefore, for a compact space $Y$ we answered  to Problem~\ref{prb1} only in the  non-scattered case.

In the present paper we solve Problem~\ref{prb1} in the case where $Y$ is a scattered compact. The central role in our investigation plays the notion of a stable pair of Hahn. A pair $(g,h)$ of functions on $X$ is called \emph{a stable pair of Hahn} if there exists a sequence of continuous functions $u_n$ on $X$ such that $h(x)=\min\limits_{n\in\mathbb{N}}u_n(x)$ and $g(x)=\max\limits_{n\in\mathbb{N}}u_n(x)$ for any $x\in X$. Evidently, every stable pair of Hahn is a countable pair of Hahn, and hence it is a pair of Hahn. We start our research with some simple observation (Theorem~\ref{t1}) that for any countable compact $Y$ and an arbitrary $X$ and for any separately continuous function $f$ on $X\times Y$ the extremal sections $(\wedge_f,\vee\!_f)$ form a stable pair of Hahn.  In Section~\ref{secStablePairs} we also characterize stable pairs in terms of functions of the first stable Baire class. Next in Section~\ref{sec_alphaT} we give an example of a separately continuous function $f$ on the square $\alpha T^2$ of Alexandroff's compactification $\alpha T$ of an uncountable discrete space $T$ such that the extremal sections $\wedge_f$, $\vee\!_f$ are not Baire one and prove a version of Theorem~\ref{t1} for a separable space $X$ and $Y=\alpha T$. In the next section we obtain two general versions of Theorem~\ref{t1}. We prove that for any separately continuous function $f$ on the product of a separable space $X$ and a scattered compact $Y$ the pair $(\wedge_f,\vee\!_f)$ is a stable pair of Hahn (Theorem~\ref{t4}). In the case where both $X$ and $Y$ are compact we replace the separability with the countable chain condition of $X$ or $Y$ (Theorem~\ref{t5}). In Sections~\ref{secCDSRS}~--~\ref{secConstrWithGiwenPH} we construct a separately continuous function with a given extremal sections. We prove that for any stable pair of Hahn $(g,h)$ on the product of a topological space $X$ and an infinity completely regular space $Y$ there exists a separately continuous function $f$ on $X\times Y$ such that $\wedge_f=g$ and $\vee\!_f=h$ (Theorem~\ref{t3}).

\section{Stable pairs of Hahn}\label{secStablePairs}

Let $X$ be a topological space. A pair  $(g,h)$ of functions $g,h:X\to\overline{\mathbb{R}}$ is called  \emph{a stable pair of Hahn}, if there exists a sequence of continuous functions $u_n:X\to\overline{\mathbb{R}}$ such that $g(x)=\min\limits_{n\in\mathbb N}u_n(x)$ and $h(x)=\max\limits_{n\in\mathbb N}u_n(x)$ for any $x\in X$. Obviously, every stable pair of Hahn is a countable pair of Hahn. For a function $f\colon X\times Y\to Z$ we always use standard notations $f^x(y)=f_y(x)=f(x,y)$, where $x\in X$, $y\in Y$.

The following theorem shows the importance of stable pairs of Hahn for the studying of the extremal sections of separately continuous functions on the product of a topological space and a scattered compact.

\begin{theorem}\label{t1} Let $X$ be a topological space, $Y$ be a compact countable topological space and $f:X\times Y\rightarrow\overline{\mathbb{R}}$ be a separately continuous function. Then $(\wedge_{f},\vee_{f})$ is a stable pair of Hahn.
\end{theorem}
\begin{proof} Let $Y=\{y_{n}: n\in\mathbb{N}\}$ and $g=\wedge_{f}$, $h=\vee_{f}$. Denote $u_{n}=f_{y_{n}}$, that is $u_{n}(x)=f(x,y_{n})$ for any $x\in X$. Since $Y$ is compact and $f^{x}$ is continuous, we conclude
$$g(x)=\inf\limits_{y\in Y}f^{x}(y)=\min\limits_{y\in Y}f^{x}(y)=\min\limits_{y\in Y}f_{y}(x)=
\min\limits_{n\in\mathbb{N}}f_{y_{n}}(x)=\min\limits_{n\in\mathbb{N}}u_{n}(x)$$
and, similarly,
$$h(x)=\sup\limits_{y\in Y}f^{x}(y)=\max\limits_{y\in Y}f^{x}(y)=\max\limits_{y\in Y}f_{y}(x)=
\max\limits_{n\in\mathbb{N}}f_{y_{n}}(x)=\max\limits_{n\in\mathbb{N}}u_{n}(x)$$
\end{proof}

Next we are going to prove some characterization of stable pairs of Hahn with the help of the stable convergence. Recall that a sequence of functions $f_n\colon X\to Y$  \emph{stably converges} to a function $f:X\to Y$ (write $f_n\mathop{\longrightarrow}\limits^d f$) if for any $x\in X$ there is $n\in\mathbb{N}$ such that $f_k(x)=f_n(x)$ for any $k\ge n$. That is  $f_n\mathop{\longrightarrow}\limits^d f$ means the pointwise convergence with respect to the discrete topology on $Y$. We say that a function $f$ is of the \textit{first stable Baire class} if there are continuous functions $f_n:X\to Y$, $n\in\mathbb{N}$, such that $f_n\mathop{\longrightarrow}\limits^d f$.

We start with some generalization of the Tietze-Urysohn theorem on the extension of a continuous function with given estimations, which is  corollary of the Tong theorem \cite{T}. Recall that a set $A$ is called \emph{functionally closed} in $X$ if there is a continuous function $\alpha:X\to[0;1]$ such that $A=\alpha^{-1}(0)$.

\begin{proposition}\label{t2} Let $X$  be  a topological space, $Y$  be  a functionally closed subspace of $X$, $(g,h)$  be  a countable pair of Hahn on $X$ and $f_{0}:X\rightarrow\overline{\mathbb{R}}$ be a continuous function such that $g(y)\leq f_{0}(y)\leq h(y)$ for any $y\in Y$. Then there exists a continuous function $f:X\rightarrow\overline{\mathbb{R}}$ such that $f(y)=f_{0}(y)$ for $y\in Y$ and $g(x)\leq f(x)\leq h(x)$ for any $x\in X$.
\end{proposition}
\begin{proof} Since there exists an increasing homeomorphism $\overline{\mathbb{R}}$ on $\mathbb{I}=[0;1]$, we may assume that $f_0,g,h:X\to\mathbb{I}$ (see \cite[Proposition 1]{MK}). Let $u,v:X\rightarrow\mathbb{I}$,
$$u(x)=\left\{
            \begin{array}{cl}
              f_{0}(x),& x\in Y \\
              g(x),& x\in X\setminus Y
            \end{array}
          \right.
,\ \ \ \
v(x)=\left\{
            \begin{array}{cl}
              f_{0}(x),& x\in Y \\
              h(x),& x\in X\setminus Y
            \end{array}
          \right.
$$
Then $g(x)\leq u(x)\leq v(x)\leq h(x)$ for any $x\in X$. Let us prove that $(u,v)$ is  a countable pair of Hahn. Since $(g,h)$ is a countable pair of Hahn, there exist sequences continuous functions $g_n,h_n:X\to[0;1]$ such that
$$
\lim\limits_{n\to\infty}g_n(x)=g(x)
\ \ \ \text{and}\ \ \
\lim\limits_{n\to\infty}h_n(x)=h(x),
$$
and  $g_n(x)\ge g_{n+1}(x)$ and $h_n(x)\le h_{n+1}(x)$ for each $x\in X$.
Since $Y$ is functionally closed, there exists a continuous function $\varphi:X\to[0;1]$ such that $Y=\varphi^{-1}(0)$.
Put for any $x\in X$
$$
u_n(x)=\max\Big\{g_n(x),f_0(x)-n\varphi(x)\Big\},
\ \ \
v_n(x)=\min\Big\{h_n(x),f_0(x)+n\varphi(x)\Big\}.
$$
It is clearly that $u_n$ and $v_n$ are continuous and $u_n(x)\ge u_{n+1}(x)$ and $v_n(x)\le v_{n+1}(x)$ for any $x\in X$. Let us show that
$$
\lim\limits_{n\to\infty}u_n(x)=u(x)
\ \ \ \text{and}\ \ \
\lim\limits_{n\to\infty}v_n(x)=v(x),\ \ \ x\in X.
$$
Firstly, consider the case $x\in Y$. Then $\varphi(x)=0$. Hence,
$$
u_n(x)=\max\Big\{g_n(x),f_0(x)\Big\}
\ \ \ \text{and}\ \ \
v_n(x)=\min\Big\{h_n(x),f_0(x)\Big\}.
$$
Therefore,
$$
\lim\limits_{n\to\infty}u_n(x)
=\max\Big\{\lim\limits_{n\to\infty}g_n(x),f_0(x)\Big\}
{=}\max\Big\{g(x),f_0(x)\Big\}=f_0(x)=u(x).
$$
In the similar way,
$$
\lim\limits_{n\to\infty}v_n(x)
=\min\Big\{\lim\limits_{n\to\infty}h_n(x),f_0(x)\Big\}
{=}\min\Big\{h(x),f_0(x)\Big\}=f_0(x)=v(x).
$$

Let us consider the case $x\notin Y$. Then $\varphi(x)>0$, and so $n\varphi(x)\to+\infty$ as $n\to\infty$. Choose $n_0\in\mathbb{N}$ such, that $n\varphi(x)>1$ for $n\ge n_0$. Hence, 
$$f_0(x)-n\varphi(x)\le0\le g_n(x)\ \ \text{and}\ \  h_n(x)\le 1\le f_0(x)+n\varphi(x).$$
Then $u_n(x)=g_n(x)$ and $v_n(x)=h_n(x)$ for any $n\ge n_0$. Thus,
$$
\lim\limits_{n\to\infty}u_n(x)=\lim\limits_{n\to\infty}g_n(x)=g(x)=u(x)
$$
and
$$
\lim\limits_{n\to\infty}v_n(x)=\lim\limits_{n\to\infty}h_n(x)=h(x)=v(x).
$$

The Tong theorem \cite{T}  yields that there exists a continuous function $f\colon X\rightarrow\mathbb{I}$ such that $u(x)\leq f(x)\leq v(x)$ for any $x\in X$. Thus, $g(x)\leq u(x)\leq f(x)\leq v(x)\leq h(x)$ on $X$ and $u(y)\leq f(y)\leq v(y)=f_{0}(y)=u(y)$ for any $y\in Y$. Therefore, $f(y)=f_{0}(y)$ for each $y\in Y$.
\end{proof}

\begin{corollary}\label{l6} Let $X$  be  a perfectly normal space, $Y$  be  a closed  subspace of $X$, $(g,h)$  be a pair of Hahn on $X$ and $f_{0}:Y\rightarrow\overline{\mathbb{R}}$ be a continuous function such that $g(y)\leq f_{0}(y)\leq h(y)$ for any $y\in Y$. Then there exists a continuous function $f:X\rightarrow\overline{\mathbb{R}}$ such that $f(y)=f_{0}(y)$ for any $y\in Y$ and $g(x)\leq f(x)\leq h(x)$ for any $x\in X$.
\end{corollary}
\begin{proof}
By the Tietze-Urysohn theorem \cite[2.1.6.]{En} there exists a continuous extension $f_1:X\to\overline{\mathbb{R}}$ of the function $f_0$. Moreover, the Tong theorem \cite{T} implies that $(g,h)$ is a countable pair of Hahn. It remains to apply Proposition~\ref{t2} to the pair $(g,h)$ and the function $f_1$.
\end{proof}

Now we pass to the characterization of stable pairs of Hahn. A function $f:X\to Y$ is said to be \emph{$\sigma$-continuous} if there exists a sequence of closed subsets $X_n$ of $X$ such that $f|_{X_n}$ is continuous for each $n\in\mathbb{N}$.

\begin{proposition}
Let $X$  be  a topological space and $g,h:X\to\overline{\mathbb{R}}$. Then the following conditions are equivalent:
\begin{itemize}
\item[$(i)$]  $(g,h)$ is a stable pair of Hahn;
\item[$(ii)$]  $(g,h)$ is a countable pair of Hahn and $g$ and $h$ are functions of the first stable Baire class;
\end{itemize}
Moreover, if $X$ is normal then these properties are equivalent to the following two properties.
\begin{itemize}
\item[$(iii)$]  $(g,h)$ is a pair of Hahn and $g$ and $h$ are functions of the first stable Baire class;
\item[$(iv)$]  $(g,h)$ is a pair of Hahn and $g$ and $h$ are $\sigma$-continuous;
\end{itemize}
\end{proposition}
\begin{proof} $(i)\Rightarrow(ii)$. Let $(g,h)$ be a stable pair of Hahn.
Obviously, $(g,h)$ is a countable pair of Hahn. By the definition, there exists a sequence of continuous functions $u_n:X\to\overline{\mathbb{R}}$ such that
$$g(x)=\min\limits_{n\in\mathbb N}u_n(x)
\ \ \ \text{and}\ \ \
h(x)=\max\limits_{n\in\mathbb N}u_n(x)
\ \ \ \text{for any }x\in X.
$$
Consider the functions $g_n,h_n:X\to\overline{\mathbb{R}}$,
$$
h_n(x)=\max\limits_{k=1,\dots,n}u_k(x)
\ \ \ \text{and}\ \ \
g_n(x)=\min\limits_{k=1,\dots,n}u_k(x)
\ \ \ \text{for any }x\in X.
$$
Fix $x\in X$ and choose $n_x,m_x\in\mathbb{N}$ such that ${u_{n_x}(x)=g(x)}$ and ${u_{m_x}(x)=h(x)}$. Put $k_x=\max\{n_x,m_x\}$. If $n\ge k_x$, then $g_n(x)=g(x)$ and $h_n(x)=h(x)$. Therefore, $g_n\mathop{\longrightarrow}\limits^d g$ and $h_n\mathop{\longrightarrow}\limits^d h$. Thus, $g$ and $h$ are functions of the first stable Baire class.

$(ii)\Rightarrow(i)$. By \cite[Theorem 3]{BKMM} we conclude that there exist sequences of functionally closed sets $A_n$ and $B_n$  and continuous functions $g_n,h_n:X\to\overline{\mathbb{R}}$ such that  $g_n|_{A_n}=g|_{A_n}$ and $h_n|_{B_n}=h|_{B_n}$ for any $n\in\mathbb{N}$ and $\bigcup\limits_{n=1}^\infty A_n=\bigcup\limits_{m=1}^\infty B_m=X$. By Proposition~\ref{t1} for any $n\in\mathbb N$ there exist continuous functions $u_{2n-1},u_{2n}:X\to\overline{\mathbb{R}}$  such that $g_n|_{A_n}=u_{2n-1}|_{A_n}$, $h_n|_{B_n}=u_{2n}|_{B_n}$ and $g(x)\le u_{2n-1}(x),u_{2n}(x)\le h(x)$ for any $x\in X$. Show, that
$$g(x)=\min\limits_{k\in\mathbb N}u_k(x)
\ \ \ \text{and}\ \ \
h(x)=\max\limits_{k\in\mathbb N}u_k(x)
\ \ \ \text{for any }x\in X.
$$
Fix $x\in X$. Then there are $n\in\mathbb N$ and $m\in\mathbb N$ such that $x\in A_n$ and $x\in B_m$. Therefore, $u_{2n-1}(x)=g_n(x)=g(x)$ and $u_{2m}(x)=h_m(x)=h(x)$. Since $g(x)\le u_k(x)\le h(x)$ for any $x\in X$, we obtain the needed equalities.

$(iii)\Leftrightarrow(iv)$. It follows from \cite[Theorem 4]{BKMM}

$(ii)\Rightarrow(iii)$. Evidently.

$(iii)\Rightarrow(ii)$. It follows from \cite[Theorem 2 and 3]{MK}.
\end{proof}

Note, that in the previous theorem in $(ii)$ we cannot omit the condition of countability of the pair of Hahn $(g,h)$. Indeed, the examples of functions from \cite[Example 1 and 2]{MK} are, in fact, functions of the first stable Baire class but they are not the limits of monotone sequences of continuous functions.

\section{The space $\alpha T$}\label{sec_alphaT}

\begin{lemma} \label{l7}
Let $\alpha T=T\cup\{\infty\}$  be  Alexandroff's compactification  of a discrete space $T$ and $f:\alpha T\rightarrow \mathbb{R}$  be a Baire one function. Then there exists a countable set $S\subseteq T$ such that $f$ is constant on $\alpha T\setminus S$.
\end{lemma}
\begin{proof}
Firstly, suppose that $f$ is continuous. Recall that a neighborhood of $\infty$ in $\alpha T$ is the complement of a finite subset of $T$.
Therefore, for any $n\in \mathbb{N}$ there exists a finite set
$E_n\subseteq T$ such that $f(\alpha T\setminus E_n)\subseteq(f(\infty)-\tfrac{1}{n};f(\infty)+\tfrac{1}{n})$.  Put $S=\bigcup\limits_{n=1}^{\infty} E_n$. Obviously, that $S$ is a countable and
$$f(\alpha T\setminus S)=f\bigg(\bigcap\limits_{n=1}^{\infty}(\alpha T\setminus E_n)\bigg)\subseteq\bigcap\limits_{n=1}^{\infty}f\big((\alpha T\setminus E_n)\big)
$$
$$
\subseteq\bigcap\limits_{n=1}^{\infty}\big(f(\infty)-\tfrac{1}{n};f(\infty)+\tfrac{1}{n}\big)=\big\{f(\infty)\big\}$$
Therefore, $f$ is constant on $\alpha T\setminus S$.

Now let us consider the case where $f$ is Baire one. Then there exists a sequence of continuous functions  $f_n:\alpha T\rightarrow \mathbb{R}$ such that $f_n(x)\rightarrow f(x)$ for any $x\in\alpha T$. By the previous part, for any $n$ there exists a countable set $S_n\subseteq T$ such that $f_n$ is constant on $\alpha T\setminus S_n$. Put $S=\bigcup\limits_{n=1}^{\infty} S_n$. Then $S$ is a countable set too. Let $x\in\alpha T\setminus S=\bigcap\limits_{n=1}^{\infty}(\alpha T\setminus S_n)$. Then for any $n\in \mathbb{N}$ we have that $x\in\alpha T\setminus S_n$, and so, $f_n(x)=f_n(\infty)$. Therefore, $$f(x)=\lim_{n\rightarrow\infty}f_n(x)=\lim_{n\rightarrow\infty}f_n(\infty)=f(\infty)$$
for any $x\in\alpha T\setminus S$.
\end{proof}

\begin{example} Let $X=Y=\alpha T=T\cup\{\infty\}$  be  Alexandroff's compactification of an uncountable discrete space $T$. Then there exists a separately continuous function $f:X\times Y\rightarrow \mathbb{R}$ such that $\vee_f$ and $\wedge_f$ are not Baire one.
\end{example} \label{tv1}
\begin{proof}
Since $\big|\{0,1,2\}\times T\big|=3\cdot|T|=|T|$, there exists a bijection $\varphi\colon\{0,1,2\}\times T\rightarrow T$. Set $T_i=\varphi(\{i\}\times T)$ for $i=0,1,2$. Then $T=T_0\sqcup T_1\sqcup T_2$ and $|T_0|=|T_1|=|T_2|=|T|$.
Define  $f:X\times Y\rightarrow \mathbb{R}$,
$$f(x,y)=\left\{
         \begin{array}{rl}
           1,& \text{if }x=y\in T_1 \\
           -1,& \text{if }x=y\in T_2\\
           0,& \text{otherwise},
         \end{array}
         \right.\ \ \ \text{ for any }(x,y)\in X\times Y
$$
For each $x\in X=\alpha T$ the function $f^x$ is equal to zero everywhere  except at most at one point of $T$. Then $f^x$ is continuous. Analogously, we show that  $f_y$ is continuous for any $y\in Y=\alpha T$. Therefore, $f$ is a  separately continuous function.

Denote $g=\wedge_f$ and $h=\vee_{\!f}$. For any $x\in T_1$ we have that $f(x,x)=1$, and then   $h(x)=1$. If $x\notin T_1$ then $f(x,y)\leq0$ for any $y$ and $f(x,\infty)=0$. So, $h(x)=0$.
Let   $\chi_A:\alpha T\to \{0,1\}$, $\chi_A(t)=1$ if $x\in A$ and $\chi_A(t)=0$ if $x\notin A$ for any $A\subseteq \alpha T$.
Therefore, $h=\chi_{T_1}$. Analogously, we show that $g=-\chi_{T_2}$. Thus, for any countable set $S$ we have that $T_i\nsubseteq S$ for any $i=0,1,2$. Hence, $g$ and $h$ are not constant on $\alpha T\setminus S$. Thus, by Lemma~\ref{l7} $g$ and $h$ are not Baire one.
\end{proof}

\begin{proposition}\label{tv2}
 Let $X$  be  a separable topological space, $Y=\alpha T$ and $f:X\times Y\rightarrow\mathbb{R}$ be  a separately continuous function. Then $(\wedge_f,\vee_f)$ is a stable pair of Hahn.
\end{proposition}
\begin{proof}
Let $A$ be a countable dense subset of $X$. By Lemma~\ref{l7} for any $x\in A$ there is a countable set $S_x\subseteq T$ such that $f^x$ is constant on $Y\setminus S_x$. Let $S=\bigcup\limits_{x\in A}S_x$. Then $S$ is a countable set and $f^x$ is constant on $Y\setminus S$ for any $x\in A$.

Fix $x\in X$. Since $\overline{A}=X$, there exists a net $(x_m)_{m\in M}$ in $A$ such that $x_m\rightarrow x$. Therefore, for any $y\in Y\setminus S$ we have that
$$f^x(y)=f_y(x)=\lim_{m\in M}f_y(x_m)=\lim_{m\in M}f^{x_m}(y)
$$
$$=\lim_{m\in M}f^{x_m}(\infty)=
\lim_{m\in M}f_{\infty}(x_m)=f_{\infty}(x)=f^x(\infty).$$ Thus, $f^x$ is constant on $Y\setminus S$ for any $x\in X$.

Obviously, $Y_1=S\cup \{\infty\}$ is Alexandroff's compactification of the discrete space $S$ in the case where $S$ is an infinity set. Thus, $Y_1$ is a countable compact. Set $f_1=f|_{X\times Y_1}$.  So, Theorem~\ref{t1} yields that $(\wedge_{f_1},\vee_{f_1})$ is a stable pair of Hahn.  Since $f(x,y)=f(x,\infty)$ for any $x\in X$ and $y\in Y\setminus S$, we conclude that
$$\wedge_{f}(x)=\inf_{y\in Y}f(x,y)=\min\{\inf_{y\in S}f(x,y),\inf_{y\in Y\setminus S}f(x,y)\}$$
$$=\min\{\inf_{y\in S}f(x,y),f(x,\infty)\}=\inf_{y\in S\cup \{\infty\}}f(x,y)=\inf_{y\in Y_1}f_1(x,y)=\vee_{f_1}(x)$$
Similarly, we show that $\vee_{f}=\vee_{f_1}$. Thus, $(\wedge_{f},\vee_{f})=(\wedge_{f_1},\vee_{f_1})$ is a stable pair of Hahn.
\end{proof}

\section{Scattered compacts}

Recall that a topological space is called \textit{scattered} if any non-empty subset $A$ of $X$ has an isolated point in $A$. Denote
$$
A^{(0)}=A,\ \ \ A^{(1)}=\Big\{x\in A:x\in \overline{A\setminus\{x\}}\Big\}
$$
and for any ordinal $\alpha>1$ we define recursively
$$
A^{(\alpha)}=\Big(\bigcap\limits_{\xi<\alpha}A^{(\xi)}\Big)^{(1)}.
$$ 
Note that $A^{(1)}$ is closed in $A$. Therefore, $A^{(\alpha)}$ is closed in $A$ for any $\alpha$.
Since $\big(X^{(\alpha)}\big)$ is a decreasing transfinite sequence, there exists an ordinal $\alpha$ such that $A^{(\alpha)}=A^{(\alpha+1)}$. The \textit{scattered rank} is the ordinal
$$
r(A)=\min\big\{\alpha:A^{(\alpha)}=A^{(\alpha+1)}\big\}.
$$
Let $X$ be a scattered space, $\alpha=r(X)$ and $E=X^{(\alpha)}$. So, $E^{(1)}=X^{(\alpha+1)}=E$. Therefore, $E$ has no any isolated points, and then $X^{(\alpha)}=E=\emptyset$. Conversely, let $X^{(\alpha)}=\emptyset$ and $A$ be a subset of $X$ without isolated points. Then $A^{(1)}=A$. Therefore, $A=A^{(\alpha)}\subseteq X^{(\alpha)}=\emptyset$. Thus, we have proved that $X$ is scattered if and only if $X^{(r(X))}=\emptyset$ (see, \cite[Note 1.2]{KR})
\begin{proposition}\label{tv3}
Let $X$  be  a first countable scattered compact. Then $X$ is countable.
\end{proposition}
\begin{proof}
We are going to prove this proposition using the transfinite induction by the scattered rank $\alpha=r(X)$. If $\alpha=0$ then $X^{(0)}=\emptyset$, and so $X=\emptyset$ and then $|X|=0<\aleph_0$. If $\alpha=1$ then $X^{(1)}=\emptyset$. Hence, $X$ is a discrete space. Since $X$ is compact, we obtain that $X$ is  finite, that is $|X|<\aleph_0$. 

Suppose that for some ordinal $\alpha>1$ any first countable scattered compact $K$ of the rank $r(K)<\alpha$ is countable. Consider some scattered compact $X$ with $r(X)=\alpha$ and prove that $|X|\leq\aleph_0$.
We have $X^{(\alpha)}=\emptyset$. Denote $F=\bigcap\limits_{\xi<\alpha}X^{(\xi)}$. Then, by the definition of the scattered rank, $X^{(\alpha)}=F^{(1)}=\emptyset$. Thus, the closed set $F$ has no limit point, and then $|F|<\aleph_0$. But $X$  is a first countable Hausdorff space. Therefore, every singleton in $X$ is of type $G_\delta$. Thus, $F$  is of type $G_\delta$ as well.

Consider a sequence of open sets $G_n$ such that $\bigcap\limits_{n=1}^\infty G_n=F$, and put $F_n=X\setminus G_n$. Then $F_n$ is a closed subset of $X$, and so, $F_n$ is a scattered compact.

Let us prove that $r(F_n)<\alpha$. Since $X^{(\xi)}$ is closed and $\bigcap\limits_{\xi<\alpha}X^{(\xi)}=F\subseteq G_n=X\setminus F_n$, we conclude that the open sets $U_\xi=X\setminus X^{(\xi)}$, $\xi<\alpha$, cover the compact set $F_n$. Then there exist $\xi_1,\xi_2,..., \xi_m<\alpha$ such that $F_n\subseteq\bigcup\limits_{i=1}^m U_{\xi_i}$. Let $\xi_0=\max\{\xi_1,..., \xi_m\}$. Since $X^{(\xi)}$'s decrease, we have that $U_\xi$'s  increase. Hence,
$$F_n\subseteq\bigcup\limits_{i=1}^m U_{\xi_i}=U_{\xi_0}=X\setminus X^{(\xi_0)}.$$ Then $F^{(\xi_0)}_n\subseteq X^{(\xi_0)}$ and $F^{(\xi_0)}_n\subseteq F_n\subseteq X\setminus X^{(\xi_0)}$. Thus, $F^{(\xi_0)}_n=\emptyset$ and then $r(F_n)\leq\xi_0<\alpha$.
So, by the inductive assumption we have that $F_n$ is countable. But
$$X=F\cup(X\setminus F)=F\cup\Big(X\setminus\bigcap\limits_{n=1}^\infty G_n\Big)= F\cup\Big(\bigcup\limits_{n=1}^\infty X\setminus G_n\Big)=F\cup \bigcup\limits_{n=1}^\infty F_n.$$ Therefore, $X$ is a countable set.
\end{proof}

\begin{theorem} \label{t4}
Let $X$ be a separable topological space, $Y$ be a scattered compact and $f:X\times Y\rightarrow\mathbb{\overline{R}}$ be a separately continuous function. Then $(\wedge_f,\vee_f)$ is a stable pair of Hahn.
\end{theorem}
\begin{proof}
Since there exists an increasing homeomorphism from $\mathbb{\overline{R}}$ to $[0;1]$, without loss of the generality we may and do assume that $f\colon X\times Y\rightarrow[0;1]\subseteq\mathbb{R}$. Set $g=\wedge_f$, $h=\vee_f$ and prove that $(g,h)$ is a stable pair of Hahn.

Let $T$  be  a countable dense subset of $X$. Consider $\psi:Y\rightarrow\mathbb{R}^X$,  $\psi(y)(x)=f(x,y)$ for any $x\in X$, $y\in Y$.  It is  wellknown that the separate continuity of $f$ implies that $\psi:Y\rightarrow C_p(X)$ is a continuous mapping. Put $\widetilde{Y}=\psi(Y)$. Then $\widetilde{Y}$ is compact. Moreover, \cite[1.9(c)]{KR} implies that $\widetilde{Y}$ is a scattered compact. Define $h:\widetilde{Y}\rightarrow\mathbb{R}^T$,  $h(\tilde{y})=\tilde{y}|_T$. Since $\widetilde{Y}\subseteq C_p(X)$ and $\overline{T}=X$, we have that $h$ is a continuous injection. Hence,  $h:\widetilde{Y}\rightarrow h(\widetilde{Y})$ is a homeomorphism. But the set $T$ is countable. Therefore, $\mathbb{R}^T$ is a metrizable space. Thus,
$h(\widetilde{Y})$ is a metrizable scattered compact. So, Proposition~\ref{tv3}  yields that $\widetilde{Y}$ is a countable set.

 Let $\tilde{f}:X\times\widetilde{Y}\rightarrow\mathbb{R}$,  $\tilde{f}(x,\tilde{y})=\tilde{y}(x)$ for any $x\in X$,  $\tilde{y}\in\widetilde{Y}$. Since $\widetilde{Y}\subseteq C_p(X)$, $\tilde{f}$ is a separately continuous function. Set $\tilde{g}=\wedge_{\tilde{f}}$ and $\tilde{h}=\vee_{\tilde{f}}$. Then by Theorem~\ref{t1} we conclude that $(\tilde{g},\tilde{h})$ is a stable pair of Hahn.
But for any $x\in X$ we have
$$g(x)=\inf_{y\in Y}f(x,y)=\inf_{y\in Y}\psi(y)(x)=\inf_{\tilde{y}\in \widetilde{Y}}\tilde{y}(x)=\inf_{\tilde{y}\in \widetilde{Y}}\tilde{f}(x,\tilde{y})=\wedge_{\tilde{f}}(x)=\tilde{g}(x)$$
where $\tilde{y}=\psi(y)$. Similarly, we show that $h(x)=\tilde{h}(x)$ for any $x\in X$. Thus, $(g,h)=(\tilde{g},\tilde{h})$ is a stable pair of Hahn.
\end{proof}

\begin{theorem} \label{t5}
Let $X$ be a compact, $Y$ be  a scattered compact  such that $X$ or $Y$ has the countable chain condition, $f:X\times Y\rightarrow\mathbb{\overline{R}}$ be a separately continuous function and $g=\wedge_f$, $h=\vee_f$. Then $(g,h)$ is a stable pair of Hahn.
\end{theorem}
\begin{proof}  Firstly, let us consider the case where $X$ has the countable chain condition and $f:X\times Y\rightarrow[0;1]$. Define $\varphi:X\rightarrow C_p(Y)$, $\varphi(x)(y)=f(x,y)$ for any $x\in X$, $y\in Y$. Obviously, $f$ is a continuous mapping. Denote $\widetilde{X}=\varphi(X)$. Since   $X$ has the countable chain condition, we have that $\widetilde{X}$ has the countable chain condition as well. But $\widetilde{X}\subseteq C_p(Y)$ and $Y$ is a compact. Therefore,  $\widetilde{X}$  is  an Eberlein compact with the countable chain condition. Thus, $\widetilde{X}$ is a separable space by \cite[Proposition 5.2.]{MMI}

Let $\tilde{f}:\widetilde{X}\times Y\rightarrow\mathbb{\overline{R}}$, $\tilde{f}(\tilde{x},y)=\widetilde{x}(y)$ for any $\tilde{x}\in\widetilde{X}$ and $y\in Y$. Since $\tilde{f}$ is a separately continuous function, Theorem~\ref{t4} yields that the functions $\tilde{g}=\wedge_{\tilde{f}}$, $\tilde{h}=\vee_{\tilde{f}}$ form a stable pair of Hahn. Fix $x\in X$ and let $\tilde{x}=\varphi(x)$. Then
$$g(x)=\wedge_f(x)=\inf_{y\in Y}f(x,y)=\inf_{y\in Y}\varphi(x)(y)$$
$$=\inf_{y\in Y}\tilde{x}(y)=\inf_{y\in Y}\tilde{f}(\tilde{x},y)=\wedge_{\tilde{f}}(\tilde{x})=\tilde{g}(\tilde{x}).$$
Thus, $g=\tilde{g}\circ\varphi$. Analogously, we prove that $h=\tilde{h}\circ\varphi$. It is easy to show that $(g,h)$ is a stable pair of Hahn, because the pair $(\tilde{g},\tilde{h})$ is a stable pair of Hahn and $\varphi$ is  continuous surjection of $X$ onto $\widetilde{X}$.

Now let us consider the case where $Y$ has the countable chain condition and $f:X\times Y\rightarrow[0,1]\subseteq\mathbb{R}$. Define $\psi:Y\rightarrow C_p(X)$, $\psi(y)(x)=f(x,y)$ for any $x\in X$, $y\in Y$. Therefore, $\widetilde{Y}=\psi(Y)$ is an Eberlein compact  with the countable chain condition. Thus, $\widetilde{Y}$ is a metrizable scattered compact  by \cite[Proposition 5.2.]{MMI} and then it is countable by Proposition~\ref{tv3}.

Consider a separately continuous function $\tilde{f}:X\times\widetilde{Y}\rightarrow\mathbb{R}$, $\tilde{f}(x,\tilde{y})=\tilde{y}(x)$ for any $x\in X$ and $\tilde{y}\in\tilde{Y}$. Then it is easy to see that $\wedge_{\tilde{f}}=\wedge_f=g$ and $\vee_{\tilde{f}}=\vee_f=h$. Thus,  Theorem~\ref{t1} implies that $(g,h)=(\tilde g,\tilde h)$  is a  stable pair of Hahn.
\end{proof}

\section{Construction of disjoint sequences in a regular space}\label{secCDSRS}

\begin{lemma} \label{l1} Let $X$  be  a regular  space and $G$  be  an infinity open subset of  $X$. Then there exists a sequence of non-empty open sets $U_{n}\subseteq G$, $n\in\mathbb{N}$, such that $\overline U_n\cap \overline U_m=\emptyset$ for any $n\neq m$.
\end{lemma}
\begin{proof}
If every point of $G$ is isolated, then we pick a sequence of distinct points $x_{n}\in G$ and put $U_{n}=\{x_{n}\}$.

Let us consider the case where there exists a non-isolated point  $x_{0}$ of $G$. Therefore, $W\setminus \{x_{0}\}\neq\emptyset$ for any neighborhood $W$ of $x_{0}$. Set $W_{1}=G$. Then there exists a point $x_{1}\in W_{1}\setminus \{x_{0}\}$. Since $X$ is a regular space and $V_{1}=W_{1}\setminus \{x_{0}\}$ is a neighborhood of $x_{1}$, there exists an open neighborhood  $U_{1}$ of $x_{1}$ with $\overline U_{1}\subseteq V_{1}$. Put $W_{2}=G\setminus \overline U_{1}$.


Suppose that for some $n>1$ we have already constructed point $x_{k}\in X$ and its open neighborhood $U_{k}\subseteq G$ in $X$ for any $k<n$ such that
\begin{itemize}
\item[$(1)$] $\overline U_{k}\cap \overline U_{j}=\emptyset$ for all $k,j<n$ such that $k\neq j$;
\item[$(2)$] $x_{0}\notin \bigcup\limits_{k<n}\overline U_{k}$.
\end{itemize}
Let us construct a point $x_{n}$ and its open neighborhood $U_{n}$ with the same properties. By $(2)$ we have that $W_{n}=G\setminus\bigcup\limits_{k<n}\overline U_{k}$ is a neighborhood of  $x_{0}$. Then there exists a point $x_{n}\in W_{n}\setminus \{x_{0}\}$. Therefore, $V_{n}=W_{n}\setminus \{x_{0}\}$ is a neighborhood of $x_{0}$ in $X$. Thus, the regularity of $X$ implies that there is an open neighborhood $U_{n}$ of $x_{n}$ such that $\overline U_{n}\subseteq V_{n}$. Therefore,
$$\overline U_{k}\cap \overline U_{j}=\emptyset\ \ \text{for all } k,j<n
\ \ \ \text{and}\ \ \
x_{0}\notin \bigcup\limits_{k\leq n}\overline U_{k}.$$
So, we have constructed the needed sequence.
\end{proof}

\begin{lemma}\label{l2} Let $X$  be  an infinity regular space. Then there exists a disjoint sequence $(G_n)_{n=1}^\infty$ of infinity open sets $G_{n}$ in $X$.
\end{lemma}
\begin{proof} Choose $(U_{n})_{n=1}^{\infty}$ by lemma~\ref{l1}. Consider a disjoint partition of $\mathbb{N}$ into infinity sets $N_{k}\subseteq\mathbb{N}$ and put $G_{k}=\bigcup\limits_{n\in N_{k}}U_{n}$.
\end{proof}

\section{Construction of special functions in a regular space}

We start with the proof of the continuity of a series of continuous functions with disjoint supports. The \emph{support} of a function $f:X\rightarrow\mathbb{R}$ is the set
$\mathrm{supp} f=\big\{x\in X:f(x)\neq0\big\}.$

\begin{lemma}\label{l3} Let $X$  be  a topological space, $(U_{n})_{n=1}^{\infty}$  be a sequence of non-empty open sets in $X$ such that $\overline U_{n}\cap\overline U_{m}=\emptyset$ for any $n\neq m$ and $f_{n}:X\rightarrow\mathbb{R}$ be  continuous functions such that $\mathrm{supp} f_{n}\subseteq U_{n}$ and $\varepsilon_{n}=\sup\limits_{x\in U_{n}}|f_{n}(x)|\to 0$ as $n\rightarrow\infty$. Then the function $f=\sum\limits_{n=1}^{\infty} f_{n}$ is continuous on $X$.
\end{lemma}
\begin{proof}
Let us show the continuity of $f$ at a point $x_{0}\in X$.
Fix $\varepsilon>0$. The disjointness of $\big(\overline U_{n}\big)_{n=1}^{\infty}$ implies that there is $n\in\mathbb{N}$ such that $x_{0}\notin\overline U_{k}$ for any $k\neq n$.
Since $\varepsilon_{k}\rightarrow 0$, there is $m>n$ such that $\varepsilon_{k}<\tfrac{\varepsilon}{3}$ for each $k\geq m$.
Then
$V=X\setminus \bigcup\limits_{k< m,k\neq n}\overline U_k$
is an open neighborhood of $x_{0}$. But $f_{n}$ is continuous at $x_{0}$. Then there exists a neighborhood $U$  of $x_{0}$ such that $U\subseteq V$ and $|f_{n}(x)-f_{n}(x_{0})|<\tfrac{\varepsilon}{3}$ for any $x\in U$. Since $U_{k}\cap U=\emptyset$ for each $k<m$ and $k\neq n$, we have that  for any $x\in U$
$$f(x)=f_{n}(x)+\sum\limits_{k\geq m}f_{k}(x).$$
The disjointness  of the supports $\mathrm{supp} f_{k}\subseteq U_{k}$ yields that the sum $\sum\limits_{k\geq m} f_{k}(x)$ has at most one nonzero summand. Then there exists $k_{x}\geq m$ such that
$$f(x)=f_{n}(x)+f_{k_{x}}(x).$$
Put $k=k_{x}$ and $l=k_{x_{0}}$. So, $k,l\geq m$ and
$$f(x)=f_{n}(x)+f_{k}(x)\text{ and }f(x_{0})=f_{n}(x_{0})+f_{l}(x_{0}).$$
Therefore,
$$|f(x)-f(x_{0})|\leq|f_{n}(x)-f_{n}(x_{0})|+|f_{k}(x)|+|f_{l}(x_0)|<\tfrac{\varepsilon}{3}+\tfrac{\varepsilon}{3}+\tfrac{\varepsilon}{3}=\varepsilon$$
Thus, $f$ is continuous at $x_{0}$.
\end{proof}

Now we pass to the construction of special functions in completely regular spaces.

\begin{lemma}\label{l4} Let $X$  be  a completely regular space and $G$ be an infinite open subset of $X$. Then there exists a sequence $(x_{n})_{n=1}^{\infty}$ in $G$ and a continuous function $f:X\rightarrow [-1,1]$ such that $\mathrm{supp} f\subseteq G$, $f(x_{2k-1})=\tfrac{1}{k}$ and ${f(x_{2k})=-\tfrac{1}{k}}$ for any $k\in\mathbb{N}$.
\end{lemma}
\begin{proof} By Lemma~\ref{l1} there is  a sequence of non-empty open sets $U_{n}\subseteq G$ such that $\overline U_{n}\cap \overline U_{m}=\emptyset$ if $n\neq m$. For any $n\in\mathbb{N}$ we pick a point  $x_{n}\in U_{n}$. Since $X$ is a completely regular space, for any $n\in\mathbb{N}$ there is a continuous function $f_{n}:X\rightarrow [0,1]$ such that $f_{n}(x_{n})=1$ and $\mathrm{supp} f_{n}\subseteq U_{n}$. Put
$$f(x)=\sum\limits_{k=1}^{\infty}\tfrac{1}{k}f_{2k-1}(x)-\sum\limits_{k=1}^{\infty}\tfrac{1}{k}f_{2k}(x).$$
For any $n\in \mathbb N$ there are $k_n\in\mathbb N$ and $i_n\in\{0,1\}$ such that ${n=2k_n-i_n}$. Therefore,  
$$f(x)=\sum\limits_{n=1}^{\infty}\tfrac{(-1)^{i_n}}{k_n}f_{n}(x),\ \ \ x\in X.$$
So, $f(x_n)=\tfrac{(-1)^{i_n}}{k_n}$ and then 
${\varepsilon_n=\sup\limits_{x\in U_{n}}|f(x)|=\tfrac{1}{k_n}\rightarrow 0}$. Thus, by Lemma~\ref{l3} the function $f$ is continuous on $X$ and $$\mathrm{supp} f\subseteq\bigcup\limits_{n=1}^{\infty}
\mathrm{supp} f_{n}\subseteq\bigcup\limits_{n=1}^{\infty}U_{n}\subseteq G.$$
Moreover,  $f(x_{2k-1})=\tfrac{1}{k}$ and $f(x_{2k})=-\tfrac{1}{k}$ for any $k$.
\end{proof}

\section{Construction of functions related with pairs of Hahn}

\begin{lemma}\label{l5} Let $X$  be  a topological space, $Y$  be a completely regular topological space, $G$ be an infinite open subset of $Y$, $g,h:X\rightarrow\mathbb{R}$ be continuous functions such that $g\leq 0\leq h$ and $A$ be a functionally closed subset of $X$. Then there exists a separately continuous function $f:X\times Y\rightarrow\mathbb{R}$ such that $\mathrm{supp} f\subseteq(X\setminus A)\times G$ and
$$
\min\limits_{y\in G}f(x,y)=g(x)\text{ and }\max\limits_{y\in G}f(x,y)=h(x)\text{ for any }x\in X\setminus A.
$$
\end{lemma}
\begin{proof} Since $A$ is a functionally closed set, there exists a continuous function $\alpha:X\rightarrow[0;1]$ such that $A=\alpha^{-1}(0)$.
Using Lemma~\ref{l4} for the space $Y$ and the set $G$ we find a sequence $(y_{n})_{n=1}^{\infty}$ in $G$ and a continuous function $\beta:Y\rightarrow[0;1]$ such that $\mathrm{supp} \beta\subseteq G$, $\beta(y_{2n-1})=\tfrac{1}{n}$ and $\beta(y_{2n})=-\tfrac{1}{n}$ for any $n\in\mathbb{N}$.

Recall  that \emph{the Schwartz function} $\mathrm{sp}:\mathbb{R}^{2}\rightarrow\mathbb{R}$ is defined by
$$\mathrm{sp}(s,t)=\left\{
            \begin{array}{cl}
              \tfrac{2st}{s^2+t^2},& \text{ if }(s,t)\neq(0,0); \\
              0,& \text{ if } (s,t)=(0,0);  \\
            \end{array}
          \right.\ \ \ \text{for any } s,t\in\mathbb R
$$
It is wellknown, $\mathrm{sp}$ is a separately continuous function. Put
$$
\varphi(s,t)=
\left\{
  \begin{array}{cl}
    1, & \hbox{if }|\mathrm{sp}(s,t)|\ge\frac12;  \\
    2|\mathrm{sp}(s,t)|, & \hbox{if } |\mathrm{sp}(s,t)|<\frac12;\\
  \end{array}
\right.\ \ \ \text{for any }s,t\in\mathbb R.
$$
Obviously, $\varphi:\mathbb{R}^{2}\rightarrow[0;1]$ is a separately continuous function.
Define a function $f:X\times Y\rightarrow\mathbb{R}$ by
$$f(x,y)=\left\{
            \begin{array}{cr}
              g(x)\varphi\big(\alpha(x),\beta(y)\big),&\hbox{if } \beta(y)<0; \\
              h(x)\varphi\big(\alpha(x),\beta(y)\big),&\hbox{if } \beta(y)\geq0;  \\
            \end{array}
          \right.\ \ \ \text{for any }(x,y)\in X\times Y.
$$
Observe that $\beta(y)=0$ implies
$g(x)\varphi\big(\alpha(x),\beta(y)\big)=0=h(x)\varphi\big(\alpha(x),\beta(y)\big)$ for any $x\in X$ and $y\in Y$. Therefore,
 $f$ is separately continuous.

Let us prove that  $\mathrm{supp} f\subseteq(X\setminus A)\times G$.
If $x\in A$, then $\alpha(x)=0$. Therefore, for any $t\in\mathbb{R}$ 
$$
\varphi\big(\alpha(x),t\big)=2\mathrm{sp}\big(\alpha(x),t\big)=2\mathrm{sp}(0,t)=0.
$$Thus, $f(x,y)=0$ for any $x\in A$.   On the other hand, if $y\notin G$ then  $\beta(y)=0$, because $\mathrm{supp} \beta\subseteq G$. But for any $s\in\mathbb{R}$ 
$$
\varphi\big(s,\beta(y)\big)=2\mathrm{sp}\big(s,\beta(y)\big)=2\mathrm{sp}(s,0)=0.
$$ 
Hence  $f(x,y)=0$ for any
$y\notin G$.

Fix $x\in X\setminus A$ and show that
$\min\limits_{y\in G}f(x,y)=g(x)$ and $\max\limits_{y\in G}f(x,y)=h(x)$.
Since $A=\alpha^{-1}(0)$, $\alpha(x)\in(0,1]$. Then  $\tfrac{1}{n+1}\le\alpha(x)\leq\tfrac{1}{n}$ for some $n\in\mathbb{N}$. Therefore,
$$\mathrm{sp}\big(\alpha(x),\tfrac{1}{n}\big)=\frac{2\alpha(x)\tfrac{1}{n}}{\alpha(x)^{2}+\tfrac{1}{n^{2}}}
\geq\frac{2\cdot\tfrac{1}{n+1}\cdot\tfrac{1}{n}}{\tfrac{1}{n^{2}}+\tfrac{1}{n^{2}}}=
\frac{2\cdot\tfrac{1}{n+1}\cdot\tfrac{1}{n}}{2\cdot\tfrac{1}{n^{2}}}=\frac{n}{n+1}\geq\frac{1}{2}.$$
Thus, $\big|\mathrm{sp}\big(\alpha(x),\pm\tfrac{1}{n}\big)\big|=\big|\pm \mathrm{sp}\big(\alpha(x),\tfrac{1}{n}\big)\big|\ge\frac12$. So,
$\varphi\big(\alpha(x),\pm\frac1n\big)=1$.
Since $\beta(y_{2n-1})=\tfrac{1}{n}>0$ and $\beta(y_{2n})=-\frac1n<0$, we have that
$$f(x,y_{2n-1})=h(x)\ \ \ \text{ and }\ \ \ f(x,y_{2n})=g(x).$$
But $0\le\varphi(s,t)\leq 1$ for any $s,t\in\mathbb{R}$.  Then $g(x)\le f(x,y)\le h(x)$ for each $y\in G$. Since $y_{2n},y_{2n-1}\in G$, we obtain
$$\max\limits_{y\in G}f(x,y)=h(x)\ \ \ \text{ and }\ \ \ \min\limits_{y\in G}f(x,y)=g(x).$$
for any $x\in X\setminus A$.\end{proof}

\section{Construction of separately continuous functions by a given stable pair of Hahn}\label{secConstrWithGiwenPH}

\begin{theorem}\label{t3} Let  $X$  be  a topological space, $Y$  be  an infinite completely regular topological space and $(g,h)$  be  a stable pair of Hahn on $X$. Then there exists a separately continuous function  $f:X\times Y\rightarrow\overline{\mathbb{R}}$ such that 
\begin{equation}\label{equ:(0)}
	\min\limits_{y\in Y}f(x,y)=g(x)\ \ \ \text{and}\ \ \ \max\limits_{y\in Y}f(x,y)=h(x)\ \ \ \text{for any}\  x\in X.
\end{equation}
\end{theorem}
\begin{proof} Let $\mathbb I=[0;1]$ and $\mathbb{H}=\mathbb{I^{\mathbb{N}}}$ be the Hilbert cube with the product topology. Therefore, $\mathbb H$ is a metrizable compact.
Consider continuous functions $u_n:X\rightarrow\mathbb{I}$ such that
\begin{equation}\label{equ:(1)}
	g(x)=\min\limits_{n\in \mathbb{N}}u_{n}(x)\ \ \ \text{and}\ \ \  h(x)=\max\limits_{n\in \mathbb{N}}u_{n}(x)\ \ \text{for any }  x\in X
\end{equation}
Then the function $\varphi:X\rightarrow\mathbb{H}$,
$$\varphi(x)=\big(u_{n}(x)\big)_{n=1}^{\infty}\text{ for any } x\in X,$$ 
is continuous as well. Denote 
$$\mathbb{H}_0=\Big\{t=(t_{n})_{n=1}^{\infty}\in \mathbb{H}: \text{there exist} \max\limits_{n\in \mathbb{N}}t_{n} \text{ and }\min\limits_{n\in \mathbb{N}}t_{n}\Big\}.$$
 The space $\mathbb{H}_0$ is a separable  metrizable space as a subspace of  $\mathbb H$. Let 
 $$g_{1}(t)=\min\limits_{n\in \mathbb{N}}t_{n}\ \ \text{ and }\ \  h_{1}(t)=\max\limits_{n\in \mathbb{N}}t_{n}\ \ \ \text{ for any }t=(t_n)_{n=1}^\infty\in\mathbb H_0.
 $$ 
 Obviously, $(g_{1},h_{1})$ is a stable pair of Hahn on $\mathbb{H}_0$, because the coordinate projections  $p_n:\mathbb{H}_0\rightarrow\mathbb{I}$, $p_n(t)=t_n$, are continuous and $g_{1}(t)=\min\limits_{n\in \mathbb{N}}p_n(t)$, $h_{1}(t)=\max\limits_{n\in \mathbb{N}}p_n(t)$ for any $t=(t_n)_{n=1}^\infty\in\mathbb H_0$. Moreover, for any $x\in X$
\begin{align*}
	g(x)&=\min\limits_{n\in\mathbb{N}}u_n(x)=g_1\Big(\big(u_n(x)\big)_{n=1}^{\infty}\Big)=g_1\big(\varphi(x)\big),\\
	h(x)&=\max\limits_{n\in\mathbb{N}}u_n(x)=h_1\Big(\big(u_n(x)\big)_{n=1}^{\infty}\Big)=h_1\big(\varphi(x)\big).
\end{align*}
Therefore, $g=g_1\circ\varphi$ and $h=h_1\circ\varphi$.

It is enough to construct a separately continuous function $f_{1}:\mathbb{H}_0\times Y\rightarrow\mathbb{I}$ such that $g_{1}=\wedge_{f_{1}}$ and $h_{1}=\vee_{f_{1}}$.
Indeed, suppose that we have constructed such a function $f_{1}$. Set  $f(x,y)=f_{1}(\varphi(x),y)$ for any $x\in X$ and $y\in Y$. Then 
 $$f^x(y)=f_1^{\varphi(x)}(y)\ \ \ \text{and}\ \ \ f_y(x)=(f_1)_y(\varphi(x))=\big((f_1)_y\circ\varphi\big)(x)$$
 for any $x\in X$ and $y\in Y$.
Hence, $f^x$ and $f_y$ are continuous.
Thus, ${f:X\times Y\rightarrow\mathbb{I}}$ is separately continuous. Moreover, for any $x\in X$
\begin{align*}
	\wedge_f(x)&=\inf\limits_{y\in Y}f(x,y)=\inf\limits_{y\in Y}f_1\big(\varphi(x),y\big)=\wedge_{f_1}\big(\varphi(x)\big)=g_1\big(\varphi(x)\big)=g(x),\\
	\vee_f(x)&=\sup\limits_{y\in Y}f(x,y)=\sup\limits_{y\in Y}f_1\big(\varphi(x),y\big)=\vee_{f_1}\big(\varphi(x)\big)=h_1\big(\varphi(x)\big)=h(x).
\end{align*}

Therefore, without loss of the generality we may assume that $X$ is a metrizable space and $g,h:X\rightarrow \mathbb I$. By the Tong theorem \cite{T} there exists a continuous function $\theta:X\rightarrow\mathbb{R}$ such that $g(x)\leq\theta(x)\leq h(x)$. We may assume that $\theta(x)=0$ on $X$. Indeed, if $\theta\ne 0$ then we may consider the  functions $\tilde{g}=g-\theta$, $\tilde{h}=h-\theta$ and, so, $\tilde{g}\leq0\leq \tilde{h}$.  If we find a separately continuous function $\tilde{f}:X\times Y\rightarrow\mathbb{R}$ such that $\wedge_{\tilde{f}}=\tilde{g}$ and $\vee_{\tilde{f}}=\tilde{h}$ then the function $f(x,y)=\tilde{f}(x,y)+\theta(x)$, $(x,y)\in X\times Y$, has the required properties.

Thus, let us assume $g(x)\leq0\leq h(x)$ on $X$. 
Put 
$$A_n=\big\{x\in X:u_n(x)=g(x)\big\}\text{ and } B_n=\big\{x\in X:u_n(x)=h(x)\big\}.$$ 
Then $g|_{A_n}=u_n|_{A_n}$ and $h|_{B_n}=u_n|_{B_n}$, and so, $g|_{A_n}$ and $h|_{B_n}$ is continuous for any $n\in\mathbb N$.
Let us prove that the sets $A_n$ and $B_n$ are closed in $X$. Obviously, $g-u_n\le 0$ is upper semicontinuous.  Thus,  
$$A_n=\{x\in X:u_n(x)\leq g(x)\}=(g-u_n)^{-1}([0;+\infty))$$ is closed in $X$. Analogously, we show that  $B_n$'s are closed as well.
By (\ref{equ:(1)}) we conclude  $X=\bigcup\limits_{j=1}^{\infty}A_j=\bigcup\limits_{k=1}^{\infty}B_k$. So,
$$
X=\bigcup\limits_{j=1}^{\infty}A_j\cap\bigcup\limits_{k=1}^{\infty}B_k=\bigcup\limits_{j,k=1}^{\infty}(A_j\cap B_k)=\bigcup\limits_{j,k=1}^{\infty}C_{j,k},$$
where $C_{j,k}=A_j\cap B_k$. The functions $g|_{C_{j,k}}$ and $h|_{C_{j,k}}$ are continuous as the restrictions of the continuous functions $g|_{A_j}$ and $h|_{B_k}$ on the set $C_{j,k}$.
Put $F_0=\emptyset$ and $F_n=\bigcup\limits_{j,k=1}^{n}C_{j,k}$ for any $n\in\mathbb N$. Then $F_n$ is closed, $F_{n-1}\subseteq F_n$ for any $n\in\mathbb{N}$ and $\bigcup\limits_{n=1}^{\infty}F_n=\bigcup\limits_{j,k=1}^{\infty}C_{j,k}=X$. 

Consider continuous functions   $\tilde{g}_n=g|_{F_n}$ and $\tilde{h}_n=h|_{F_n}$  for any $n\in\mathbb N$. Note, that $(g,0)$ and $(0,h)$ are pairs of Hahn and $$g(x)=\tilde{g}_n(x)\leq 0\leq\tilde{h}_n(x)=h(x)$$ for any $x\in F_n$. Using Corollary~\ref{l6} twice for a subspace $Y=F_n$, pairs $(g,0)$ or $(0,h)$ and functions  $f_0=\tilde g_{n}$ or $f_0=\tilde h_n$, we construct continuous functions $g_n,h_n:X\rightarrow\mathbb{R}$ such that for any $n\in\mathbb{N}$
\begin{align}
	\label{equ:(2)}
	g(x)\leq g_n(x)\leq 0\leq h_n(x)\leq h(x)\ \ \ \ \ \ \ \ \  &\ \text{ for any }x\in X,\\
	\label{equ:(3)}
	g_n(x)=\tilde{g}_n(x)=g(x),\ \ h_n(x)=\tilde{h}_n(x)=h(x) &\ \text{ for any }x\in F_n
\end{align}

By Lemma ~\ref{l2} there exists a sequence of infinite open disjoint sets $G_n$ in $Y$. Put $E_n=F_n\setminus F_{n-1}$ for any $n\in\mathbb{N}$. Since $X$ is metrizable,  $F_n$ is functionally closed. Using  Lemma~\ref{l5} for the sets $A=F_{n-1}$, $G=G_n$ and the functions  $g=g_n$ and $h=h_n$, we conclude that there is a separately continuous function $f_n:X\times Y\rightarrow\mathbb{R}$ such that $\mathrm{supp}f_n\subseteq(X\setminus F_{n-1})\times G_n$ and
\begin{equation}\label{equ:(4)}
	\min\limits_{y\in G_n}f_n(x,y)=g_n(x),\ \  \max\limits_{y\in G_n}f_n(x,y)=h_n(x)\ \ \text{ for any }x\in X\setminus F_{n-1}
\end{equation}
Since $G_n$'s are disjoint, we obtain the disjointness of the supports $\mathrm{supp}f_n$. Thus,  the function $f:X\times Y\rightarrow\mathbb{R}$,
$$f(x,y)=\sum\limits_{n=1}^{\infty}f_n(x,y), \text{ for any }(x,y)\in X\times Y,$$
is welldefined.

First of all, we prove that $f$ is separately continuous. Fix $x\in X$. Since $X=\bigcup\limits_{n=1}^{\infty}F_n=\bigcup\limits_{n=1}^{\infty}E_n$, there exists $n\in\mathbb{N}$ such that $x\in E_n=F_n\setminus F_{n-1}$. Then  $x\in F_n\subseteq F_{m-1}$ for any $m>n$. Thus, $f_m(x,y)=0$ for any $m>n$ and $y\in Y$. And so,
$$f(x,y)=\sum\limits_{k=1}^{n}f_k(x,y),\ \ \ \text{for any } y\in Y.$$
Since $f_k$'s are separately continuous, we obtain the continuity of $f$ with respect to the second variable.

 Consider $y\in Y$. Since $G_n$'s are disjoint, there exists $n\in\mathbb{N}$ such that $y\notin G_k$ for  any $k\neq n$. But $\mathrm{supp}f_k\subseteq(X\setminus F_{k-1})\times G_k$. So, we have  $f_k(x,y)=0$ for any $k\neq n$ and $x\in X$. Thus,
$$f(x,y)=\sum\limits_{k=1}^{\infty}f_k(x,y)=f_n(x,y)\ \ \ \text{for any } x\in X.$$
 Hence, $f$ is continuous with respect to the first variable.

Finally, let us prove (\ref{equ:(0)}).
Fix $x\in X$.  As previously, we find $n\in\mathbb{N}$ with $x\in E_n$ and then
$$f(x,y)=\sum\limits_{k=1}^{n}f_k(x,y)\ \ \ \text{for any } y\in Y$$
 Consider $y\in Y$ and find $m\leq n$ such that $y\notin G_k$ for any $k\neq m$ and $k\leq n$. Then we have that
$f(x,y)=f_m(x,y)$.
Since $F_{m-1}\subseteq F_{n-1}$ and $x\notin F_{n-1}$, we obtain $x\in X\setminus F_{m-1}$. Thus, by (\ref{equ:(2)}) and (\ref{equ:(4)}) we have that
\begin{align*}
	g(x)&\leq g_m(x)=\min\limits_{z\in Y} f_m(x,z)\leq f_m(x,y)\\
	&=f(x,y)\leq\max\limits_{z\in Y} f_m(x,z)= h_m(x)\leq h(x)
\end{align*}
Therefore, 
\begin{equation}\label{equ:(5)}
	g(x)\leq f(x,y)\leq h(x)\ \ \  \text{for any } y\in Y.
\end{equation}

Using (\ref{equ:(4)}) we find points $y_1, y_2 \in G_n$ such that
$f_n(x,y_1)=g_n(x)$ and $f_n(x,y_2)=h_n(x)$.
 But $f_k(x,y_1)=0$ and $f_k(x,y_2)=0$ for $k\neq n$. Moreover, since $x\in F_n$, we conclude that $g_n(x)=g(x)$ and  $h_n(x)=h(x)$ by (\ref{equ:(3)}). Therefore,
\begin{align*}
	g(x)&=g_n(x)=f_n(x,y_1)=f(x,y_1)\ \ \ \text{and}\\ 
	h(x)&=h_n(x)=f_n(x,y_2)=f(x,y_2).
\end{align*}
Thus, the previous equalities and (\ref{equ:(5)}) give (\ref{equ:(0)}).
\end{proof}

\subsection*{Acknowledgements}
The research was supported by the University of Silesia Mathematics Department
(Iterative Functional Equations and Real Analysis program).


\end{document}